\documentclass{llncs}
\usepackage{algorithmic}
\usepackage[ruled,linesnumbered,lined,boxed,procnumbered]{algorithm2e}
\usepackage{geometry} 
\usepackage{tikz,float,comment,amsmath,amsfonts,amssymb,refcount}
\usetikzlibrary{shapes.geometric, arrows}
\usepackage{hyperref}
\newtheorem{observation}{Observation}

\newcommand{\ComplexityFont}[1]{%
{\ensuremath{\mathsf{#1}}}
}
\newcommand{\NP}{\ComplexityFont{NP}}
\renewcommand{\P}{\ComplexityFont{P}}

\author{S.M.Dhannya\and N.S.Narayanaswamy}
\institute{Department of Computer Science and Engineering, \\ IIT Madras, Chennai, India.}
\begin{document}
\title{Perfect Resolution of Strong Conflict-Free Colouring of Interval Hypergraphs}
\maketitle

\begin{abstract}
The $k$-Strong Conflict-Free ($k$-SCF, in short) colouring problem seeks to find a colouring of the vertices of a hypergraph $H$ using minimum number of colours so that in every hyperedge $e$ of $H$, there are at least $\min\{|e|,k\}$ vertices whose colour is different from that of all other vertices in $e$. In the case of interval hypergraphs, we present an exact $\P$-time algorithm for the $k$-SCF problem thus solving an open problem posed by Cheilaris et al. (2014). We achieve our results by showing that for any hypergraph a $k$-SCF colouring is a proper colouring of a related simple graph which we refer to as a \textit{co-occurrence graph}.  We then show that a co-occurrence graph is obtained by identifying an induced subgraph of a second simple graph that we introduce, which we refer to as  the \textit{conflict graph}.  For interval hypergraphs, we show that each co-occurrence graph and the conflict graph are perfect graphs.  
This property plays a crucial role in our polynomial time algorithm.
Secondly, we show that for an interval hypergraph, the $1$-SCF colouring number is the minimum partition of its intervals into sets such that each set has an exact hitting set (a hitting set in which each interval is hit exactly once).
\end{abstract}

\section{Introduction}\label{sec:intro}
Let $H = (\mathcal{V},\mathcal{E})$ be a hypergraph and let $\mathbb{N}$ denote the set of non-negative integers.  A colouring function $C: \mathcal{V} \rightarrow \mathbb{N}$ is a \textit{$k$-SCF colouring} of $H$ if for every hyperedge $e \in \mathcal{E}$, there are at least $\min\{|e|, k\}$ non-zero unique colours in $e$. In other words, for each hyperedge $e \in \mathcal{E}$, there exists distinct non-zero colours $c_1,c_2,\cdots,c_{\min\{|e|, k\}}$ such that for $i = 1,2,\ldots,\min\{|e|, k\} $, $|\{v \mid v \in e, C(v) = c_i\}| = 1$. The natural computational problem is to find a $k$-SCF colouring of $H$ using minimum number of colours. We refer to the number of colours used in an optimum $k$-SCF colouring of $H$ as its \textit{$k$-SCF colouring number} and is denoted by $\chi_{cf}^k(H)$.  In the  $k$-SCF colouring problem, the algorithm is presented with an input instance in which all vertices are initially coloured with colour $0$. 
The goal is to modify the colour of some vertices to a non-zero colour such that the resulting colouring is a $k$-SCF colouring. 
The $k$-SCF colouring  problem was first studied by Cheilaris et al. \cite{CPLGARSS2014} and is a generalized variant of a well-studied hypergraph colouring problem known as the \textit{Conflict-Free colouring} problem. A Conflict-Free (CF, in short) colouring is a vertex colouring of a hypergraph such that every hyperedge $e$ has at least one vertex whose colour is different from that of every other vertex in $e$. The minimum number of colours needed to CF colour a hypergraph is called its \textit{CF colouring number} and is denoted by $\chi_{cf}(H)$. Motivated by a frequency assignment problem in cellular networks, the CF colouring problem was introduced by Even, Lotker, Ron and Smorodinsky \cite{ELRS2003}. In mobile communication networks, one must assign frequencies to base stations such that every client that comes under the transmission range of multiple base stations can associate itself to a unique base station without any interference from another base station. The transmission range of various base stations may be viewed as geometric regions such as discs in the 2-dimensional plane and the colours as the frequency of the transmitting tower. CF colouring problem also finds applications in other areas like RFID (Radio Frequency Identification) networks, robotics and computational geometry (see the survey by Somorodinsky \cite{Sm2013}).
\subsection*{Past work in CF colouring} \label{sec:CurTechniques}
Since simple graphs are hypergraphs in which each hyperedge has exactly two vertices, it follows that
the CF colouring problem is a generalization of the proper colouring problem on simple graphs.  Thus the CF colouring problem is $\NP$-complete. 
There have been many bounds on the CF colouring numbers  in geometric hypergraphs and hypergraphs induced by neighbourhoods in simple graphs. The survey due to Smorodinsky \cite{Sm2013}  presents a general framework for CF colouring. This framework involves finding a proper colouring of the hypergraph in every iteration and giving the largest colour class a new colour. Smorodinsky \cite{Sm2013} showed that if for every induced sub-hypergraph $H' \subseteq H$,  the chromatic number of $H'$ is at most $p$, then $\chi_{cf}(H) \leq \log_{1+ \frac{1}{p-1}} n = O(p \log n)$, where $n = |\mathcal{V}|$. Pach and Tardos \cite{PJGT2009} showed that if $|\mathcal{E}(H)| <$ $ \binom{s}{2}$ for some positive integer $s$, and $\Delta$ is the maximum degree of vertices in $H$, then $\chi_{cf}(H)  < s$ and $\chi_{cf}(H) \leq \Delta + 1$. 
Consequently, it follows that for a hypergraph with $m$ hyperedges, $\chi_{cf}(H) = O(\sqrt{m})$. 
Even et al. \cite{ELRS2003} have studied a number of hypergraphs induced by geometric regions on the plane including discs, axis-parallel rectangles, regular hexagons, and general congruent centrally symmetric convex regions in the plane. Let $\mathcal{D}$ be a set of $n$ finite discs in $\mathbb{R}^2$. 
For a point $p \in \mathbb{R}^2$, define $r(p) = \{D \in \mathcal{D} : p \in D\}$. The hypergraph $(\mathcal{D}, \{r(p)\}_{p \in \mathbb{R}^2})$, denoted by $H(\mathcal{D})$, is called the hypergraph induced by $\mathcal{D}$. Smorodinsky showed that $ \chi_{cf}(H(\mathcal{D})) \leq \log_{4/3} n$ \cite{Sm2007,Sm2013}. Similarly, if $\mathcal{R}$ is a set of $n$ axis-parallel rectangles in the plane, then, $\chi_{cf}(H(\mathcal{R})) = O(\log^2n)$. There have been many studies on hypergraphs induced by neighbourhoods in simple graphs. Given a simple graph $G = (V,E)$, the \textit{open neighbourhood} (or simply neighbourhood) of a vertex $v \in V$ is defined as follows: $N(v) = \{u \in V | (u,v) \in E\}$. The set $N(v) \cup v$ is known as the \textit{closed neighbourhood} of $v$. Pach and Tardos \cite{PJGT2009} have shown that the vertices of a graph $G$ with maximum degree $\Delta$ can be coloured with $O(\log^{2+\epsilon} \Delta)$ colours, so that the closed neighbourhood of every vertex in $G$ is CF coloured. They also showed that if the minimum degree of vertices in $G$ is $\Omega(\log \Delta)$, then the open neighbourhood can be CF coloured with at most $O(\log^2 \Delta)$ colours. Abel et al. \cite{Abel2017} gave the following tight worst-case bound for neighbourhoods in planar graphs: three colours are sometimes necessary and always sufficient. Keller and Smorodinsky \cite{Keller2018} studied CF colourings of intersection graphs of geometric objects. They showed that the intersection graph of $n$ pseudo-discs in the plane admits a CF colouring with $O(\log n)$ colours, with respect to both closed and open neighbourhoods. Ashok et al. \cite{MAXCFC2015} studied an optimization variant of the CF colouring problem named \textsc{Max}-CFC. Given a hypergraph $H = (\mathcal{V},\mathcal{E})$ and integer $r \geq 2$, the \textsc{Max}-CFC problem is to find a maximum-sized subfamily of hyperedges that can be CF coloured with $r$ colours. They gave an exact algorithm running in $O(2^{n+m})$ time. They also studied the problem in the parametrized setting where one must find if there exists a subfamily of at least $d$ hyperedges that can be CF coloured using $r$ colours. They showed that the problem is FPT and gave an algorithm with running time $2^{O(d \log \log d + d \log r)} (n+m)^{O(1)}$. 
\subsection*{CF colouring in interval hypergraphs} A hypergraph $H_n=([n],\mathcal{I}_n)$, where $[n] = \{1,2, \ldots, n\}$  and $\mathcal{I}_n=\big\{ \{i, i+1, \ldots, ,j\} \mid i \leq j \text{ and } i,j \in [n] \big\}$ is known as a \textit{complete interval hypergraph} \cite{CPLGARSS2014}. It was shown in \cite{ELRS2003} that a complete interval hypergraph can be CF coloured using $\Theta(\log n)$ colours. Chen et al. \cite{CFKLMMPSSWW2006} presented results on an online variant of CF colouring problem in complete interval hypergraphs. In this variant, points arrive online and a point has to be assigned a colour upon its arrival such that the resulting colouring is a CF colouring of the complete interval hypergraph.  Chen et al. \cite{CFKLMMPSSWW2006} gave a greedy algorithm that uses $\Omega(\sqrt n)$ colours, a deterministic algorithm that uses $\Theta(\log^2n)$ colours and a randomized algorithm that uses $O(\log n)$ colours. A hypergraph in which the set of hyperedges is a family of intervals $\mathcal{I} \subseteq \mathcal{I}_n$ is known as an \textit{interval hypergraph}. 
One can view the CF colouring problem in interval hypergraphs as modelling the frequency assignment problem in a chain of unit discs \cite{CPLGARSS2014,ELRS2003}. Such a chain of unit discs may be viewed as corresponding to transmission ranges of base stations on approximately unidimensional networks like national highways or railway networks \cite{CPLGARSS2014}. Cheilaris et al. \cite{CPLGARSS2014} have presented as example the case of channel assignment for broadcasting in a wireless mesh network. During some steps of broadcasting, it may so happen that sparse receivers of the broadcast message are inside the transmission range of a linear sequence of transmitters. In this case, only a subset of linear sequences of discs representing the transmitters are involved \cite{Nguyen2011,Zeng2010}.
\subsection*{$k$-SCF colouring in interval hypergraphs} Katz et al. \cite{KATZ} gave a $\P$-time approximation algorithm for $1$-SCF colouring an interval hypergraph with an approximation ratio 4. Cheilaris et al. \cite{CPLGARSS2014} improved the approximation ratio to 2 in their paper on $k$-SCF colouring.  In the case of interval hypergraphs, the techniques by Katz et al. \cite{KATZ}, Cheilaris et al. \cite{CPLGARSS2014} and Cheilaris and Smorodinsky \cite{CS2012} relied on an understanding of combinatorial structures that result in high $1$-SCF colouring number. 
For instance, if an interval has two disjoint intervals completely contained inside it, then it is easy to see that this structure needs at least two  colours in an optimal $1$-SCF colouring. The authors in \cite{CS2012} refer to a generalization of this  structure as $\mathcal{J}_{\rho}$ configuration, which is defined as follows. Family $\mathcal{J}_1$ contains all singleton sets of intervals. For $\rho > 1$, a set of intervals $\mathcal{I}$ is in family $\mathcal{J}_{\rho}$ if and only if it can be expressed as a union $\mathcal{I} = L \cup R \cup \iota$, where both $L, R \in \mathcal{J}_{\rho-1}$, no interval from $L$ has a common point with an interval from $R$, and interval $\iota$ includes every interval in $L$ and every interval in $R$. Cheilaris et al. \cite{CPLGARSS2014} and Cheilaris and Smorodinsky \cite{CS2012} showed that any set of intervals that contains a $\mathcal{J}_{\rho}$ configuration as a subset uses at least $\rho$ colors in any $1$-SCF colouring.  They gave a 2-approximation algorithm based on the value of $\rho$, the number of levels of containment, as the lower bound. 
In this research, we present a new lower bound for general hypergraphs and show that this lower bound is indeed tight in the case of interval hypergraphs.
We present the key definitions required to state our results and a summary of our results in the next section. 

\section{Motivation and Our Results}
We observe that, for each $\rho > 1$, there are interval hypergraphs for which $\chi^1_{cf}$ is more than $\rho$ and which do not contain a $\mathcal{J}_{\rho}$ configuration.
For instance, it can be easily verified that the examples given in Figure \ref{fig:forbidEx} need at least two colours but they do not contain a $\mathcal{J}_{2}$ configuration. 

\begin{figure}[h]
\centering
\begin{tikzpicture}[scale=0.6]
\draw(0,2)--(5.5,2) (1,2.3) node{$I_1$}; 
\draw(2,2.4)--(8,2.4) (5.5,2.7) node{$I_2$};
\draw(0,1)--(3,1) (2,0.7) node{$I_3$};
\draw(2.75,1.3)--(5,1.3) (3.75,0.9) node{$I_4$}; 
\draw(4.5,1)--(8,1) (6.25,0.7) node{$I_5$}; 
\node at (10,1){};
\node at (4,-1) {\small{(a) Type 1}};
\end{tikzpicture}
\begin{tikzpicture}[scale=0.6]
\draw(0,2)--(4.5,2) (2.5,2.3) node{$I_1$}; 
\draw(3.5,2.25)--(8,2.25) (5.5,2.55) node{$I_2$};
\draw(0,1)--(2,1) (1,0.6) node{$I_3$}; 
\draw(2.5,1)--(5.5,1) (3.5,0.7) node{$I_4$}; 
\draw(6,1)--(7.9,1) (7.25,0.6) node{$I_5$}; 
\node at (4,-1) {\small{(b) Type 2}};
\end{tikzpicture}\\
\caption{Configurations that need at least two colours}
\label{fig:forbidEx}
\end{figure}%
Combinations of these structures also lead to a large $\chi^1_{cf}$.
For instance, consider the set of intervals, as shown in  Figure \ref{fig:forbidFused}, which is obtained by `fusing' two sets of intervals of Type 1 shown in Figure \ref{fig:forbidEx}. We start with two sets $\mathcal{I}_1 = \{I_1,I_2,I_3,I_4,I_5\}$ and $\mathcal{L}_1 = \{L_1,L_2,L_3,L_4,L_5\}$. Extend $I_5$ to the right such that the right endpoint of $I_5$ is one point to the right of right endpoint of $I_2$. Note that our intervals are all finite sets of consecutive integers.  Then, fuse $\mathcal{L}_1$ to $\mathcal{I}_1$ as follows. Let $I_5$ and $L_3$ fuse to become a single interval $I_5L_3$, while ensuring that $I_2 \cap L_1 \neq \emptyset$. This new structure also needs two colours even though it does not contain Type 1 structure, Type 2 structure or a $\mathcal{J}_{2}$ configuration as a substructure.

\begin{figure}[h]
\centering
\begin{tikzpicture}[scale=0.7, every node/.style={scale=0.7}]
\draw(0,2)--(4.5,2) (1,2.2) node{$I_1$}; 
\draw(2,2.4)--(6,2.4) (3.5,2.6) node{$I_2$};
\draw(0,1)--(2.5,1) (1.5,1.2) node{$I_3$};
\draw(2.25,1.3)--(4,1.3) (3,1.5) node{$I_4$}; 
\draw(5,2.8)--(9,2.8) (6.5,3) node{$L_1$};
\draw(7,3.2)--(11,3.2) (9.5,3.4) node{$L_2$};
\draw(3.75,1)--(7.5,1) (5.25,1.2) node{$I_5L_3$}; 
\draw(7.25,1.3)--(8.75,1.3) (7.9,1.5) node{$L_4$}; 
\draw(8.5,1)--(11,1) (9.25,1.2) node{$L_5$};
\end{tikzpicture}
\caption{`Fusion' of two sets of intervals in Figure \ref{fig:forbidEx} (a)}
\label{fig:forbidFused}
\end{figure}%

Thus we conclude that combinations of structures like Type 1, Type 2  and $\mathcal{J}_{\rho}$ configuration result in an infinite family of structures that can be used to construct interval hypergraphs with arbitrarily high $\chi^1_{cf}$.  
 Hence we believe that a complete characterization or even strengthening the lower bound based on such substructures might be inherently complex. On the other hand, the lower bound that we present turns out to be valuable for interval hypergraphs.
 
 
\subsection*{Main Result}
The new lower bound that we present is indeed a tight bound for interval hypergraphs and it leads to our main result which is an optimal $\P$-time algorithm for the $k$-SCF colouring in interval hypergraphs. This algorithm indeed solves an open problem posed by Cheilaris et al. \cite{CPLGARSS2014}.
\begin{itemize}
\item[$\blacksquare$] The $k$-SCF colouring problem in interval hypergraphs can be solved in polynomial time. (Theorem \ref{thm:kCFCIntHypPolyTime})
\end{itemize}
We observe that a $k$-SCF colouring of a hypergraph can be naturally seen as the proper colouring of a related simple graph known as a {\em co-occurrence graph} which we define below.  A co-occurrence graph is defined for a $k$-SCF colouring and requires another abstraction that we refer to as a \textit{$k$-representative function}.

\noindent
For a hypergraph $H = (\mathcal{V},\mathcal{E})$, let $[\mathcal{V}]^k$ denote the set $\{X \mid (X \subseteq \mathcal{V}) \wedge (|X| \leq k)\}$. 
\begin{definition}[$k$-representative function]
 Given a $k$-SCF colouring function $C$ of $H$,
a function $t:\mathcal{E} \rightarrow [\mathcal{V}]^k$ is a $k$-representative function if for each $e \in \mathcal{E}$, $t(e)$ is a set of $\min\{|e|,k\}$ vertices in $e$ such that the colour given to any vertex in $t(e)$ by $C$ is not given to any other vertex in $e$.  We say that hyperedge $e$ is $k$-SCF coloured by the vertices in $t(e)$. 

\end{definition}

Next, we observe that any function $t:\mathcal{E} \rightarrow [\mathcal{V}]^k$ such that for each edge $e$, $t(e) \subseteq e$ and $|t(e)| =  \min\{|e|,k\}$, is a $k$-representative function obtained from some $k$-SCF colouring of $H$.  In particular, the $k$-SCF colouring is a proper colouring of the graph $\Gamma_{t}(H)$ called a co-occurrence graph, which is defined as follows.
\begin{definition}[Co-occurrence graph]
Given a function  $t:\mathcal{E} \rightarrow [\mathcal{V}]^k$ such that for each edge $e$, $t(e) \subseteq e$ and $|t(e)| =  \min\{|e|,k\}$, the co-occurrence graph of hypergraph $H$, denoted by $\Gamma_{t}(H)$, is defined as follows:
The vertex set of $\Gamma_{t}(H)$ is $R = \bigcup_{e \in \mathcal{E}} t(e)$. 
For $u,v \in R$,  $uv$ is an edge in  $\Gamma_{t}(H)$ if and only if  for some $e \in \mathcal{E}$, $\{u,v\} \cap t(e) \neq \emptyset$ and $u \in e$ and $v \in e$. 

\end{definition}
Clearly, a proper colouring of the graph $\Gamma_{t}(H)$ can be extended to a  $k$-SCF colouring of $H$ in which 
the vertices of $H$ which are not present in $\Gamma_{t}(H)$ get the $0$ colour.
Each element in $t(e)$ is known as a representative of $e$.  Wherever $H$ is implied, we use  $\Gamma_{t}$ to denote $\Gamma_{t}(H)$. 
The following result  gives the relationship between the $k$-SCF colouring number of a hypergraph and the chromatic number of its co-occurrence graphs.  
\begin{itemize}
\item[$\blacksquare$] The $k$-SCF colouring number of a hypergraph $H$ is equal to the chromatic number of a co-occurrence graph that has minimum chromatic number over all possible co-occurrence graphs of $H$. (Theorem \ref{thm:kCo-occChar}) 
\end{itemize}
%
As a consequence of Theorem \ref{thm:kCo-occChar} we have reduced the problem of finding an optimal $k$-SCF colouring to the problem of finding an optimal $k$-representative function. To find an optimal $k$-representative function we introduce a  simple graph associated with $H$, which we refer to as the \textit{conflict graph}.
\begin{definition}[Conflict graph] \label{defn:ConflictGraph}
The conflict graph of $H$, denoted by $\hat{G}(H)$, has the vertex set $V = \big\{(e,v) \mid e \in \mathcal{E}, v \in e\big\}$ and the edge set $E = E_{edge} \cup E_{colour}$, where $E_{edge}$ and $E_{colour}$ are defined as follows:
\begin{itemize}
\item $E_{edge} = \bigg\{ \big\{(e,v),(e,u)\big\} \mid \{v,u\} \subseteq e, u \neq v \bigg\}$. 
\item $E_{colour} = \bigg\{ \big\{(e,v),(g,u)\big\} \mid \{v,u\} \subseteq e \text{ or } \{v,u\} \subseteq g, u \neq v, e\neq g \bigg\}$. 
\end{itemize} 

\end{definition}
The elements of $V(\hat{G}(H))$ and $\mathcal{V}(H)$ are referred to as {\em nodes} and \textit{vertices}, respectively. A node $(e,v)$ encodes the proposition that $v$ is a representative of $e$.  In a node $(e,v)$,  we refer to $e$ as the hyperedge coordinate and $v$ as the vertex coordinate. Wherever $H$ is implied, we use $\hat{G}$ to denote $\hat{G}(H)$. The following structural properties of the conflict graphs are important in our results.
\begin{observation} \label{obs:NodesOfaVertexAreIndep}
Let $H = (\mathcal{V},\mathcal{E})$ be a hypergraph.
\begin{itemize}
\item For each vertex $v \in \mathcal{V}$, the nodes in $\hat{G}$ with $v$ as the vertex coordinate form an independent set.
\item For each hyperedge $e$ in $\mathcal{E}$, the nodes in $\hat{G}$ with $e$ as the hyperedge coordinate form a clique. 
\end{itemize}     
\end{observation}
\noindent
We define two special types of cliques in the conflict graph which are crucial to our results. 
\begin{definition}[Hyperedge Cliques and Colour Cliques]\label{defn:ConfGrph2TypesCliques}
\\
$\square$ Hyperedge Clique: A clique in a conflict graph formed by nodes having the same hyperedge coordinate. The set of hyperedge cliques in a conflict graph is denoted by $\mathcal{Q}_1$.\\
$\square$ Colour Clique: A maximal clique in a conflict graph that has at least one edge from $E_{colour}$. The set of colour cliques in a conflict graph is denoted by $\mathcal{Q}_2$.

\end{definition}
The conflict graph plays a crucial role in formulating the problem of finding an optimal $k$-representative function. For this, we require the concept of an \textit{exact-$k$-hitting set}, which is defined as follows.

\begin{definition}[Exact-$k$-hitting set] \label{defn:exactKHittingSet}
An exact-$k$-hitting set of a hypergraph $X = (U,\mathcal{S})$ is a set $U' \subseteq U$ such that for each hyperedge $s \in \mathcal{S}$, $|U' \cap s| = \min\{|s|, k\}$.

\end{definition} 
The optimal $k$-representative function is obtained from an exact-$k$-hitting set $S \subseteq V(\hat{G})$ of the hyperedge cliques in $\hat{G}$ such that the chromatic number of the subgraph of $\hat{G}$ induced by $S$ is minimized.  This gives a  new lower bound for the number of colours needed in a $k$-SCF colouring. 
\begin{itemize}
\item[$\blacksquare$] The $k$-SCF colouring number of a hypergraph $H$ is at least as large as the minimum chromatic number of the subgraph induced by an exact-$k$-hitting set of hyperedge cliques of the conflict graph.
 (Theorem \ref{thm:Co-OccConflRelation})
\end{itemize}
In the case of interval hypergraphs, we prove that this lower bound is tight, due to the fact that conflict graphs of interval hypergraphs are perfect. For interval hypergraphs, given such an exact-$k$-hitting set, one can find a $k$-SCF colouring from a co-occurrence graph in polynomial time due to two important facts: chromatic number of a perfect graph can be found in $\P$-time and co-occurrence graphs of interval hypergraphs are perfect.
\begin{itemize}
\item[$\blacksquare$] Conflict graphs and co-occurrence graphs of interval hypergraphs are perfect. (Theorem \ref{thm:ConflictPerf}, Theorem \ref{thm:Co-occPerf})
\end{itemize}
%
Finally, in interval hypergraphs, we study the relationship between $1$-SCF colouring problem and the partition of given set of intervals into parts each of which has an exact hitting set.  It is easy to see that for a hypergraph $H = (\mathcal{V},\mathcal{E})$, a $1$-SCF colouring of $H$ using at most $\ell$ non-zero colours partitions $\mathcal{E}(H)$ into $\ell$ hypergraphs such that each hypergraph has an exact hitting set. Interestingly, when $H$ is an interval hypergraph which can be partitioned into $\ell$ interval hypergraphs, each of which has an exact hitting set, then $H$ can be $1$-SCF coloured with at most $\ell$ non-zero colours. This is our characterization in Theorem \ref{thm:EHS-CF}.
\begin{itemize}
\item[$\blacksquare$] For an interval hypergraph $H$, the $1$-SCF colouring number is equal to the minimum number of parts in a partition of $H$ into interval hypergraphs each of which has an exact hitting set. (Theorem \ref{thm:EHS-CF})
\end{itemize}

\subsection{Preliminaries}\label{sec:Prelims}
In an interval $I = \{i,i+1,\ldots,j\}$, $i$ and $j$ are the \textit{left} and \textit{right endpoints} of $I$ respectively, denoted by $l(I)$ and $r(I)$, respectively. Since an interval is a finite set of consecutive integers, it follows that $|I|$ is well-defined. Throughout the paper, hypergraph $H$ is assumed to have $n$ vertices and $m$ hyperedges.

For a set $S$ of vertices in a simple graph $G$, $G[S]$ denotes the \textit{induced subgraph} of $G$ on $S$. \\
\noindent
Perfect graphs \cite{Gol2004} are very well-studied 
and many hard problems are tractable on perfect graphs. We use four well known properties of perfect graphs. 
\begin{enumerate}
\item[P1] Let $G = (V,E)$ be a perfect graph. For a given subset $V' \subseteq V$, let $G[V'] = (V',E_{V'}) $ be the subgraph induced by $V'$, where $E_{V'} = \{uv \in E \mid u,v \in V'\}$. Then, it is known from  \cite{Gol2004} that $ \omega(G[V']) = \chi(G[V'])$, where $\omega(G[V'])$ and  $\chi(G[V'])$ are, respectively, the clique number and the chromatic number of $G[V']$. Recall that the clique number of a simple graph $G$ is the size of its largest clique and the chromatic number of $G$ is the number of colours needed in an optimal proper colouring of $G$.
\item[P2] A \textit{Berge graph} is a simple graph that has neither an \textit{odd hole} nor an \textit{odd anti-hole} as an induced subgraph \cite{Berge1985,Chudnovsky2005,Chudnovsky2006,Gol2004}. An odd hole is an induced cycle of odd length that has at least 5 vertices and an odd anti-hole is the complement of an odd hole. It is known from Theorem 1.2 in \cite{Chudnovsky2006} that a graph is perfect if and only if it is Berge. 
\item[P3] The chromatic number of a perfect graph can be found in polynomial time  \cite{GROTSCHEL1984325}.
\item[P4] The maximum weighted clique problem can be solved in polynomial time in perfect graphs \cite{Grotschel1981},\cite{GROTSCHEL1984325}. 
\end{enumerate} 

\textbf{Separation Oracle:} A separation oracle is a polyomial time algorithm that given a point in $\mathbb{R}^d$, where $d$ is the number of variables in a linear program relaxation, either confirms that this point is a feasible solution, or produces a violated constraint (See Section 12.3.1 in\cite{VaziraniApproxAlgo}).  \\
We use $deg(v)$ to denote the degree of a vertex $v$.\\
All other definitions and notations used in this paper are from West \cite{West} and Smorodinsky \cite{Sm2013}.

\section{\texorpdfstring{$k$}{k}-Strong Conflict-Free Colouring}  \label{sec:k-SCF}
In Section \ref{subsec:Co-OccConflict}, we present a characterization of the $k$-SCF colouring number in terms of the minimum chromatic number of some co-occurrence graph associated with the given hypergraph. In Section \ref{subsec:Co-occConflictPerfect}, we prove important structural properties of a co-occurrence graph and the conflict graph when the underlying hypergraph is an interval hypergraph. 
\subsection{Co-occurrence Graphs and Conflict Graphs} \label{subsec:Co-OccConflict}
We first present the relationship between the $k$-SCF colouring number of a hypergraph and the chromatic number of its co-occurrence graphs. 
Recall that $\chi_{cf}^k(H)$ is the number of non-zero colours used in any optimal $k$-SCF colouring of $H$.   Define $\chi_{min}^k(H) = \min\limits_{t'} \chi(\Gamma_{t'})$ where $\chi(\Gamma_{t'})$ is the chromatic number of the co-occurrence graph $\Gamma_{t'}$ and the minimum is taken over all $k$-representative functions $t'$ of the hypergraph $H$.  
\begin{theorem} \label{thm:kCo-occChar}
For a hypergraph $H = (\mathcal{V},\mathcal{E})$ and a positive integer $1 \leq k \leq n$,  $\chi_{cf}^k(H) = \chi_{min}^k(H)$.

\end{theorem}
\begin{proof}
Let $t$ be a $k$-representative function such that $\chi_{min}^k(H)$ $= \chi(\Gamma_{t})$. Extend a proper colouring $C$ of $\Gamma_{t}$ to a vertex colouring function $C'$ of $\mathcal{V}(H)$ by assigning the colour $0$ to those vertices in $\mathcal{V}(H) \setminus R$, where $R$ is the set of representatives defined by $t$. $C'$ is a $k$-SCF colouring of $H$, that is,  for each $e \in \mathcal{E}$, the colour assigned to each vertex in $t(e)$ by $C'$ is different from the colour assigned to every other vertex in $e$. The reason for this is as follows. Let $v$ be a vertex in $e$. If $C'(v)=0$, then for each vertex $u \in t(e)$, $C'(v) \neq C'(u)$. If $C'(v)$ is non-zero and $v \notin t(e)$, then it implies that there is an $e'$ such that $v \in t(e')$. Consequently, $v \in V(\Gamma_{t})$. Since $v \in e$, there is an edge from all vertices in $t(e)$ to $v$ by definition of $\Gamma_{t}$. Similarly, if $C'(v)$ is non-zero and $v \in t(e)$, then there is an edge from $v$ to every other vertex in $t(e) \setminus v$. Further, since $C'$ is obtained from a proper colouring $C$ of $\Gamma_{t}$ it follows that $C'(v)$ is different from colour assigned to each vertex in $t(e)$ by $C'$. Thus $\chi_{cf}^k(H) \leq \chi_{min}^k(H)$. Now, we prove that $\chi_{min}^k(H) \leq \chi_{cf}^k(H)$ as follows: since a minimum $k$-SCF colouring of $H$ gives a natural $k$-representative function $t$, it follows that $\chi_{cf}^k(H) = \chi(\Gamma_{t}) \geq \chi_{min}^k(H)$.  Therefore, it follows that $\chi_{cf}^k(H) = \chi_{min}^k(H)$.  
\qed
\end{proof}

\noindent
In the following lemma, we present a relationship between the clique number (chromatic number) of a co-occurrence graph and the clique number (chromatic number) of an induced subgraph of the conflict graph.
\begin{theorem}\label{thm:Co-OccConflRelation}
Let $H = (\mathcal{V},\mathcal{E})$ be a hypergraph. Let $t$ be a $k$-representative function of $H$ and let $\hat{G}$ be the conflict graph of $H$. Then, the set $S = \{(e,u) \mid (e,u) \in \hat{G}, u \in t(e)\}$ is an exact-$k$-hitting set of hyperedge cliques, $\omega(\Gamma_{t}) \geq \omega(\hat{G}[S])$, and $\chi(\Gamma_{t}) \geq \chi(\hat{G}[S])$.  

\end{theorem}

\begin{proof}
By our premise, $t$ is a $k$-representative function and hence the set $S$, obtained from $t$ as defined above, hits every hyperedge clique exactly $k$ times. It follows from Definition \ref{defn:exactKHittingSet} that $S$ is indeed an exact-$k$-hitting set of the set of hyperedge cliques. Now we show that $\omega(\Gamma_{t}) \geq \omega(\hat{G}[S])$. 
Let $\big\{(e,u),(f,v)\big\}$ be an edge in $\hat{G}[S]$. By definition of set $S$, $u \in t(e)$ and $v \in t(f)$. We consider two cases:\\
Case 1 - when $e = f$: In this case,  $u, v \in e$ and the edge $\big\{(e,u),(f,v)\big\}$ belongs to $E_{edge}$ of $\hat{G}$. Both $u$ and $v$ are representatives of $e$. Hence, by definition of a co-occurrence graph, $uv$ is an edge in $\Gamma_{t}$. \\
Case 2 - when $e \neq f$: In this case, the only possibility is that the edge $\big\{(e,u),(f,v)\big\}$ belongs to $E_{colour}$ of $\hat{G}$. It follows that $u$ and $v$ are both present together in either $e$ or $f$. Without loss of generality, let $u,v \in e$. Since $(e,u) \in \hat{G}[S]$, it follows from the construction of $S$ that $u$ belongs to $t(e)$. Hence, $uv$ is an edge in $\Gamma_{t}$.\\
Therefore, for every edge $\big\{(e,u),(f,v)\big\}$ in $\hat{G}[S]$, there exists an edge $uv$ in $\Gamma_{t}$. It follows that for every clique in $\hat{G}[S]$, there exists a clique of same size in $\Gamma_{t}$. Hence, $\omega(\Gamma_{t}) \geq \omega(\hat{G}[S])$.
Further, given a proper colouring of $\Gamma_{t}$,  let the colour given to the node $(e,u)$ be the colour given to vertex $u$ in the proper colouring of $\Gamma_{t}$. From Observation \ref{obs:NodesOfaVertexAreIndep}, there are no edges between two nodes with the same vertex coordinate. Further, for each edge $\big\{(e,u),(f,v)\big\}$ in  $\hat{G}[S]$, the edge $uv$ is in $\Gamma_{t}$, and hence it follows that the colouring of $\hat{G}[S]$ is a proper colouring.   Thus $\chi(\Gamma_{t}) \geq \chi(\hat{G}[S])$.  Hence the theorem.
\qed
\end{proof}
As a consequence of Theorem \ref{thm:kCo-occChar} and Theorem \ref{thm:Co-OccConflRelation}, it follows that the $k$-SCF colouring number is at least the minimum chromatic number of an induced subgraph formed by an exact-$k$-hitting set of the hyperedge cliques of $\hat{G}$.

\subsection{For interval hypergraphs, the conflict graph and each co-occurrence graph is a perfect graph} \label{subsec:Co-occConflictPerfect}
First, we show that the conflict graph of an interval hypergraph is a perfect graph. In this proof, $\mu(H)$ denotes the number of vertices in $\hat{G}$. Note that $\mu(H) = \sum_{I \in \mathcal{I}} |I|$. 
\begin{theorem}\label{thm:ConflictPerf}
The conflict graph of an interval hypergraph is a perfect graph.

\end{theorem}
\begin{proof}
By property P2 of perfect graphs given in Section \ref{sec:Prelims}, we know that for each $p > 1$, induced  odd cycle $C_{2p+1}$ and its complement denoted by $\overline{C_{2p+1}}$ are forbidden induced subgraphs for perfect graphs.  
Our proof of perfectness of $\hat{G}$ is by starting with the hypothesis that the claim is false and then deriving a contradiction.
Let $H = (\mathcal{V},\mathcal{J})$ be an interval hypergraph  for which $\hat{G}$ is not perfect, and among all such interval hypergraphs, $H$ minimizes $\mu(H)$. Since $\hat{G}$ is not perfect, let us consider a minimal induced subgraph of $\hat{G}$, denoted by, say $F$ for which $\omega(F) \neq \chi(F)$.   We claim that for every interval $I \in \mathcal{J}$ such that $|I| > 1$, both the nodes $(I,l(I))$ and $(I,r(I))$ belong to $F$. The proof of this claim is by contradiction to the fact that $H$ is an interval hypergraph  that minimizes $\mu(H)$ and for which $\hat{G}$ is not perfect.  Let $I$ be an interval in $\mathcal{J}$ such that $|I| > 1$ and the node $(I,r(I)) \notin V(F)$. Consider the hypergraph $H' = (\mathcal{V},\mathcal{J}')$ where  $\mathcal{J}' = (\mathcal{J} \setminus I) \cup (I \setminus r(I))$.  Let $\hat{G}'$ denote the conflict graph of $H'$. Observe that $V(\hat{G}') = V(\hat{G}) \setminus \{(I, r(I))\}$. Since $(I,r(I)) \notin V(F)$ and $(I,r(I)) \notin V(\hat{G}')$, it follows that $F$ is an induced subgraph of $\hat{G}'$ also.  Hence it follows that $\hat{G}'$ is imperfect.  Further, $\mu(H') < \mu(H)$.    This contradicts the hypothesis that $H$ is the interval hypergraph  with minimum $\mu(H)$ for which $\hat{G}$ is imperfect.  Therefore, it follows that for each interval $I \in \mathcal{J}$, $(I,r(I))$ is a node in $F$. An identical argument shows that for each interval $I \in \mathcal{J}$, $(I,l(I))$ is also a node in $F$. Hence it follows that $\forall I \in \mathcal{J}$ such that $|I| > 1$, both the nodes $(I,l(I))$ and $(I,r(I))$ belong to $F$.
We now consider two exhaustive cases to obtain a contradiction to the known structure of $F$ which we know is either a $C_{2p+1}$ or a $\overline{C_{2p+1}}$ for some $p > 1$. \\
{\em Case 1- When $F$ is an induced odd cycle $C_{j}, j \geq 5$}:   In the following proof, we consider different cases, and in each case we conclude that three nodes of $C_j$ form a $K_3$ in $\hat{G}$.  This is a contradiction to the fact  that induced cycles of length at least 4 do not have a $K_3$, and we refer to this as {\em a contradiction} in the proof below.  \\
 We know that all the intervals $I$ such that $|I| > 1$ have both the nodes $(I,l(I))$ and $(I,r(I))$ in $F$.  Let $(I'',q)$ be a node such that  for some $I'$, $q$ is in interval $I'$ and $q$ is different from $r(I')$ and $l(I')$. Then, from the definition of $E_{edge}$ and $E_{colour}$, it follows that the 3 nodes $(I',l(I')), (I',r(I')), (I'',q)$ form a $K_3$, a contradiction.  As a consequence of this observation, it also follows that for any two nodes $(I_1,q_1)$ and $(I_2,q_2)$ in $C_j$ for which $|I_1| > 1$ and $|I_2| > 1$,  $l(I_1)$ and $l(I_2)$ are different, and  $r(I_1)$ and $r(I_2)$ are different.  
Therefore, for each node $(I,q)$ in $C_j$, $q$ is either $l(I)$ or $r(I)$ or $|I|=1$, and $q$ is the left endpoint (or right endpoint) of at most one interval, and for each interval $I'$, $q$ is not an element of $I' \setminus \{l(I'),r(I')\}$ (we call this set as the strict interior of $I'$).  \\
 From the conclusions above, the intervals of length more than 1 contribute  an even number of distinct nodes to the cycle $C_j$.   Since $C_j$ is an induced odd cycle, it follows that in $C_j$ there is at least one more node $(I'',q)$ for which $|I''| = 1$. It follows that $I''$ contains only the point $q$. 
Let $(I_1,q_1)$ and $(I_2,q_2)$ be the two neighbours of $(I'',q)$ in $C_j$. From Observation \ref{obs:NodesOfaVertexAreIndep}, it follows that $q$ is different from $q_1$ and $q_2$.  From the analysis above, it follows that $q$ and $q_1$ are endpoints of $I_1$, and $q$ and $q_2$ are endpoints of $I_2$.  Again from the conclusions above, since $l(I_1)$ and $l(I_2)$ are different, and since $r(I_1)$ and $r(I_2)$ are different,
without loss of generality, let us consider $q = l(I_1) = r(I_2)$ and $l(I_2) = q_2 < q < q_1 = r(I_1)$.  Therefore, the three nodes $(I_1, r(I_1)), (I'',q), (I_2,l(I_2))$ form a path in $F$.  We know that $(I_1,l(I_1))$ and $(I_2,r(I_2))$ are also vertices in $C_j$ which is an induced (that is, chordless) cycle.   Therefore, $(I_1,l(I_1)), (I_1, r(I_1)), (I'',q), (I_2,l(I_2)), (I_2,r(I_2))$ is a path in $C_j$.  In other words, $(I_1,q), (I_1, q_1), (I'',q), (I_2,q_2), (I_2,q)$ is an induced path of length 5 in $C_j$, since $(I_1,q)$ and $(I_2,q)$ are not adjacent, by Observation \ref{obs:NodesOfaVertexAreIndep}.  Since $C_j$ is a cycle, it has at least one another node, say $(I_3, q_3)$, which is  the second neighbour of $(I_1,l(I_1))$ in $C_j$.   We now show that all the points in $I_3$ are at least $r(I_1)$, and thus they are all larger than $q$.    By the definition of $E(\hat{G})$ we know that $I_3 \cap I_1 \neq \emptyset$.  Further, $q_3$ is an endpoint of $I_3$, and $q_3$ is not in the strict interior of $I_1$, and since $q$ is in $I_1$ and $I_2$, and as per our convention each interval corresponds to a single hyperedge in $H$, it follows that $q_3$ is different from $q$.  Consequently, it follows that $l(I_3) = r(I_1)$ and $q_3$ is $r(I_3)$.  Note that this argument includes the case when $|I_3| = 1$, in which case $r(I_1) = q_3$.  It follows that $I_3$ is an interval such that each point in $I_3$ is  at least $r(I_1)$ which is larger than $q$.\\
Therefore from the conclusions made thus far, each node in $C_j$ is one of two types: either the  hyperedge coordinate is such that all the points in the corresponding interval are at most  $q$ or the hyperedge coordinate is such that all the points in the corresponding interval are more than $q$.  In particular, $I_2$ is such that all the points are at most $q$ and $I_3$ is such that all points are more than $q$.  Since $C_j$ is an induced cycle, it follows that there are two adjacent nodes $(I_l,q_l)$ and $(I_r,q_r)$ such that all points in $I_l$ are at most $q$, and all points in $I_r$ are more than $q$.  In other words, $I_l$ and $I_r$ are two disjoint intervals, and we have concluded that  $(I_l,q_l)$ and $(I_r,q_r)$ are adjacent.  This is a contradiction to the definition of $E(\hat{G}) = E_{colour} \cup E_{edge}$.  This contradiction has been arrived at due to the assumption that there is a $C_j$ of odd length at least 5.  
Hence, in this case our hypothesis that there is a minimal $H$ for which $\hat{G}$ is not perfect is wrong.\\
\noindent
{\em Case 2- When $F$ is the complement of an odd cycle, say $\overline{C_{j}}, j \geq 5$}: Here we order the nodes in non-decreasing order of their vertex coordinate. Let the order be $(I_1,p_1), (I_2,p_2), \ldots,$ $ (I_j,p_j)$.  Since each node is adjacent to exactly $j-3$ vertices in $F$, it follows that $(I_1,p_1)$ is not adjacent to $(I_{j-1},p_{j-1})$ and $(I_j,p_j)$ in $\hat{G}$.  Similarly, $(I_j,p_j)$ is not adjacent to $(I_1,p_1)$ and $(I_2,p_2)$.  The reason is that if there is an edge between $(I_1,p_1)$ and $(I_{j-1},p_{j-1})$ then one of the two nodes is adjacent to all the nodes whose vertex coordinates are between $p_1$ and $p_{j-1}$.  Such a node will have degree $j-2$ which contradicts the fact that all the nodes in $F$ have degree $j-3$.  Since the degree of each vertex is $j-3$, it follows that $(I_1,p_1)$ is adjacent to all nodes from $(I_2,p_2)$ to $(I_{j-2},p_{j-2})$.  Similarly, $(I_j,p_j)$ is  adjacent to all nodes from $(I_3,p_3)$ to $(I_{j-1},p_{j-1})$.  Now, let us consider $(I_2,p_2)$ and $(I_{j-1},p_{j-1})$. If these two nodes are adjacent in $F$, then one of the two will have  degree at least $j-2$. Such a case cannot happen. Therefore, the nodes $(I_1,p_1),(I_j,p_j),(I_2,p_2),(I_{j-1},p_{j-1}),(I_1,p_1)$ forms an induced 4-cycle in $\overline{F}$.  $\overline{F}$ is an induced cycle $C_j, j \geq 5$, and by definition does not contain an induced cycle of length 4.  Thus our hypothesis that $\hat{G}$ contains $F$ is false.  \\
 In either case the assumption of the existence of a minimal $H$ for which $\hat{G}$ is not perfect leads to a contradiction to the known structure of perfect graphs. 
Hence, it follows that for an interval hypergraph $\hat{G}$ is perfect. 
\qed
\end{proof}
\noindent
We next prove that each  co-occurrence graph of an interval hypergraph is a perfect graph.
\begin{theorem}\label{thm:Co-occPerf}
Let $H = (\mathcal{V},\mathcal{I})$ be an interval hypergraph and $k \leq n$ be any positive integer. Let $t$ denote a $k$-representative function defined on $\mathcal{I}$. Then, the co-occurrence graph $\Gamma_{t}(H)$ is a perfect graph.

\end{theorem}
\begin{proof}
We use property P2 of perfect graphs stated in Section \ref{sec:Prelims} to prove this result. We show that $\Gamma_{t}$ does not have an odd cycle or its complement as induced subgraphs. We first show that $\Gamma_{t}$ does not have an induced cycle of length at least 5. Our proof is by contradiction. 
Assume that $F = \{p_1,p_2 \ldots p_r\}$ is an induced $C_r$-cycle for $r \geq 5$. Let the sequence of nodes in $F$ be $p_1,p_2 \ldots p_r,p_1$. Let $p_i$ be the rightmost point of $F$ on the line.  In what follows, the arithmetic among the indices of $p$ is $mod$ $r$.  
Observe that, due to cyclicity of $C$, if $i = 1$, then $i-1 = r$. Similarly, if $i = r$, then $i+1 = 1$ and $i+2 = 2$. Without loss of generality, let us assume that $p_{i-1} < p_{i+1}$, which are the two neighbours of $p_i$ in $F$. Therefore, $p_{i-1} < p_{i+1} < p_i$.  Since edge $(p_{i-1},p_i)$ is in $F$, it follows that there exists an interval $I 
\in \mathcal{I}$ for which $p_{i-1}$ belongs to $I$ and $t(I) \cap \{p_{i-1},p_i\} \neq \emptyset$. We claim that $p_i \in t(I)$ and $p_{i-1} \notin t(I)$: if $t(I)$ contains $p_{i-1}$, then $(p_{i-1},p_{i+1})$ is an edge in $\Gamma_{t}$ by definition. Therefore, $(p_{i-1},p_{i+1})$ is a chord in $F$, a contradiction to the fact that $F$ is an induced cycle.  Therefore, $t(I)$ contains $p_i$ and not $p_{i-1}$.  Further, we claim that the point $p_{i+2} < p_{i-1}$: if $p_{i+2} > p_{i-1}$, then $p_{i+2}$ belongs to the interval $I$ and by the definition of the edges in $\Gamma_{t}$, $(p_i,p_{i+2})$ is an edge in $\Gamma_{t}$.  Therefore,  $(p_i,p_{i+2})$ is a chord in $F$. This contradicts the fact that $F$ is an induced cycle. Therefore, $p_{i+2} < p_{i-1}$.   At this point in the proof we have concluded that $p_{i+2} < p_{i-1} < p_{i+1} < p_i$ and $\{p_{i-1},p_i\} \setminus p_{i-1} \in t(I)$.  Since $(p_{i+1},p_{i+2})$ is an edge in $F$, it follows that there exists an interval $J$ such that both $p_{i+1}$ and $p_{i+2}$ belong to $J$ and $t(J) \cap \{p_{i+1},p_{i+2}\} \neq \emptyset$.  Since $F$ is an induced cycle of length at least 5, for each $p \in t(J)$, $(p_{i-1},p)$ is an edge in $\Gamma_{t}$ by definition. Therefore,  $(p_{i-1},p)$ is a chord in either case,  that is when $p_{i+1} \in t(J)$ or $p_{i+2} \in t(J)$.  This contradicts the assumption that $F$ is an induced cycle of length at least 5.  Thus, $\Gamma_{t}$ cannot have an induced cycle of size at least 5.

Next, we show that $\Gamma_{t}$ does not contain complements of cycles of length $\geq 5$, ($\overline{C_r}$, $r \geq 5$) as an induced subgraph. Again, our proof is by contradiction. Assume that $F$ is an induced $\overline{C_{r}}$, $r \geq 5$ in $\Gamma_{t}$. Let $q_1, q_2, \ldots, q_r$ be the nodes of $F$. Also, let $q_1 < q_2 < \ldots < q_r$ be the left to right ordering of points on the line corresponding to vertices of $F$. Since $deg(q_i) = r-3$ for all $q_i$ in $F$, it follows that no interval $I$, such that $t(I) \cap F \neq \emptyset$, contains more than $r-2$ vertices from $F$. Otherwise, if there exists an interval $I$ such that $t(I) \cap F \neq \emptyset$ contains more than $r-2$ vertices from $F$, then for each $q \in t(I)$, $deg(q) \geq r-2$ in $F$ which is a contradiction. Therefore, there does not exist any interval that contains both $q_1$ and $q_r$. Similarly, there does not exist any interval that contains both $q_1$ and $q_{r-1}$ and any interval that contains both $q_2$ and $q_{r}$. Since $deg(q_1) = r-3$, it follows that $q_1$ must be adjacent to all vertices in $\{q_2,q_3,\ldots,q_{r-2}\}$.
Similarly, $q_r$ must be adjacent to all vertices in $\{q_3,q_4,\ldots,q_{r-1}\}$. Next, we consider the degrees of vertices $q_2$ and $q_{r-1}$ in $F$. Since they are in $F$, $q_2$ is adjacent to $q_1$ and $q_{r-1}$ is adjacent to $q_r$. 
Now, $q_2$ must be adjacent to $r-4$ more vertices. We show that $q_{2}$ is not adjacent to $q_{r-1}$. Suppose not, that is, if $q_{2}$ is adjacent to $q_{r-1}$, then there exists an interval $I$ that contains both $q_{2}$ and $q_{r-1}$ and $t(I) \cap \{q_2, q_{r-1}\} \neq \emptyset$. Then each point $q \in t(I)$ is adjacent to all points in the set $
\{\{q_2,q_3,\ldots,q_{r-1}\} \setminus q\}$. Thus, by considering the one additional edge incident on $q$ depending on whether $q = q_2$ or $q = q_{r-1}$, it follows that $deg(q) \geq r-2$, a contradiction to the fact that the degree of each vertex inside $F$ is $r-3$.   Therefore, it follows that  $(q_{2},q_{r-1})$ does not exist in $F$. It follows that in $
\overline{F}$, which we know is an induced cycle of length at least 5,   there is an induced cycle $q_1, q_{r-1}, q_2, q_r, q_1$ of length $4$.  This contradicts the structure of an induced cycle of length at least 5.   Hence, we conclude that $\Gamma_{t}$ does not have an induced cycle of length 5 or more or its complement.  Therefore $\Gamma_{t}$ is a perfect graph. 
\qed
\end{proof}
We have now set up all the machinery to formulate a linear program for the $k$-SCF colouring problem on interval hypergraphs and to solve it in polynomial time.

\section{Optimally solving \texorpdfstring{$k$}{k}-SCF colouring in interval hypergraphs} \label{sec:AlgoForK-SCF}
In this section, we present Algorithm \texttt{CFC-Intervals} (Algorithm \ref{algo:MainAlgo}) that optimally solves the $k$-SCF colouring in interval hypergraphs in $\P$-time. The important steps in Algorithm \ref{algo:MainAlgo} have been given in the flowchart in Figure \ref{fig:flowchart2}.

\tikzstyle{block} = [rectangle, draw, 
    text width=20em, text centered, rounded corners, minimum height=2em]
\tikzstyle{line} = [draw, -latex']
\tikzstyle{cloud} = [draw, ellipse, minimum height=2em]
\begin{figure}
\centering
\begin{tikzpicture}[node distance = 1.6cm, auto, scale=0.8,every  node/.style={scale=0.8}]


	\node [cloud,draw](hyper) at (-1,4) {Input};
    \node [rectangle, draw, text width=7em, text centered, rounded corners, minimum height=2em] (constructConflict) at (5,4) {Construct conflict graph $\hat{G}(H)$};
    \node [rectangle, draw, text width=7em, text centered, rounded corners, minimum height=2em] (formulateLP) at (-1,0) {On $\hat{G}$ as input, formulate LP};
	\node [rectangle, draw, text width=10em, text centered, rounded corners, minimum height=2em] (SPAlg) at (6,0) {Run Algorithm $\mathtt{SPAlg}$ to solve LP using the ellipsoid method by referring to \texttt{SPMaxWtClique}};
	\node at (10,2) [draw, rectangle, text width=10em, text centered, rounded corners, minimum height=2em] (SPMaxWtClique) {Polynomial time separation oracle \texttt{SPMaxWtClique}};
	\node at (11.5,0) [draw, rectangle, text width=10em, text centered, rounded corners, minimum height=2em] (Rounding) {On $X_{opt}$ as input, run the rounding algorithm \texttt{Rounding}};                 
	\node at (3,-3.5) [draw, rectangle, text width=15em, text centered, rounded corners, minimum height=2em] (obtainKRep) {From $X_{optI}$, obtain the $k$-representative function $t$};   
	\node at (9,-3.5) [draw, rectangle, text width=10em, text centered, rounded corners, minimum height=2em] (constructCoOcc){Construct the co-occurrence graph $\Gamma_{t}$}; 
	\node at (2,-7) [draw, rectangle, text width=10em, text centered, rounded corners, minimum height=2em] (propColourGamma){Obtain a proper colouring $C$ of $\Gamma_{t}$}; 
	\node at (7,-7) [draw, rectangle, text width=10em, text centered, rounded corners, minimum height=2em] (kSCF){Obtain a $1$-SCF colouring $C_{cf}(H)$ from $C(\Gamma_{t})$}; 
	\node [cloud,draw](output) at (5,-9) {Output};

	\node at (5.5,0.8) [draw,dotted, rectangle, text width=47em, text centered, rounded corners, minimum height=12em] (borderLP) {};
	\node at (5.5,-3.5) [draw,dotted, rectangle, text width=40em, text centered, rounded corners, minimum height=10em] (borderCoOcc) {};
	\node at (5.5,-7) [draw,dotted, rectangle, text width=35em, text centered, rounded corners, minimum height=6em] (borderColour) {};
%
    \path [line] (hyper) -- node {Interval hypergraph $H$} (constructConflict);
    \path [line] (constructConflict) -- node[above] {$\hat{G}$} (formulateLP);
    \path [line] (formulateLP) -- node {LP instance $\mathcal{B}$} (SPAlg);
    \path [line] (SPMaxWtClique) -- (SPAlg); 
    \path [line] (SPAlg) -- node {$X_{opt}$} (Rounding);
    \path [line] (Rounding) -- node[below] {$X_{optI}$} (obtainKRep);
	\path [line] (obtainKRep) -- node {$t$}  (constructCoOcc);
	\path [line] (constructCoOcc) -- node[above] {$\Gamma_{t}$} (propColourGamma);
	\path [line] (propColourGamma) -- node {$C(\Gamma_{t})$} (kSCF);
	\path [line] (kSCF) -- node {$C_{cf}(H)$} (output);
\end{tikzpicture}
\caption{Flow chart depicting key steps of the main algorithm}
\label{fig:flowchart2}   
\end{figure}


\begin{algorithm}[h]
\caption{\texttt{CFC-Intervals}(Interval hypergraph $H = (\mathcal{V},\mathcal{I})$,$k$)}
\label{algo:MainAlgo}
\begin{algorithmic}[1]
\FOR {each $v \in \mathcal{V}$}
\STATE{$C_{cf}(v) \leftarrow 0$} \hfill $\blacktriangleright$ \textit{\small{Initialize colour of each vertex to $0$}}
\ENDFOR
\STATE{$\hat{G} \leftarrow $ conflict graph of $H$}
\STATE{$X_{opt}, q_{min} \leftarrow $\texttt{SPAlg}$(\hat{G},k)$}
\hfill
$\blacktriangleright$ \textit{\small{Solve the LP to obtain $X_{opt}$ and $q_{min}$}}
\STATE{$X_{optI} \leftarrow $ \texttt{RoundingFrac}($X_{opt},\mathcal{I}$)}
\hfill
$\blacktriangleright$ \textit{\small{Round the LP solution to $0,1$ solution}}
\STATE{$S_{min} \leftarrow \{(I,u) \mid x_{I,u} = 1 \text { in solution } X_{optI}\}$}
\hfill
$\blacktriangleright$ \textit{\small{Hitting set of hyperedge cliques}}
\STATE{Define $t$ as follows: $u \in t(I) \Leftrightarrow (I,u) \in S_{min}$}
\hfill
$\blacktriangleright$ \textit{\small{$k$-representative function}}
\STATE{$\Gamma_{t} \leftarrow$ Co-occurrence graph on $t$}
\STATE{$C \leftarrow $ A proper colouring of $\Gamma_{t}$}
\FOR {each $v \in \mathcal{V}$}
\IF {$v \in V(\Gamma_{t})$}
\STATE{$C_{cf}(v) \leftarrow C(v)$}
\hfill
$\blacktriangleright$ \textit{\small{$1$-SCF colouring from proper colouring of $\Gamma_{t}$}}
\ENDIF
\ENDFOR
\RETURN $C_{cf}, q_{min}$ \hfill
$\blacktriangleright$ \textit{\small{$q_{min}$ is the $k$-SCF colouring number of $H$}}
\end{algorithmic}
\end{algorithm}

\subsection{\texorpdfstring{$k$}{k}-representative function from a hitting set of hyperedge cliques in \texorpdfstring{$\hat{G}$}{the conflict graph}} 
\label{sec:kCFCfromEHSofType1Cliques}
Let  $H = (\mathcal{V},\mathcal{I})$ be an interval hypergraph and let $\hat{G}$ be its conflict graph. Let $S_{min} \subseteq V(\hat{G})$ be an exact-$k$-hitting set of hyperedge cliques of $\hat{G}$ that hits every colour clique at most $q_{min}$ times, and let $q_{min}$ be the smallest integer for which such an $S_{min}$ exists. It follows that $S_{min}$ will have exactly $\min\{|I|, k\}$ nodes corresponding to each interval $I$. 
Define the function $t$ from $\mathcal{I}$ to $[\mathcal{V}]^k$ as follows: $t(I) = \{u \mid (I,u) \in S_{min}\} $. In Lemma \ref{lem:OmegaGtLThanQmin}, we strengthen Theorem \ref{thm:Co-OccConflRelation} for interval hypergraphs: we show  that $t$ is indeed a $k$-representative function and that the chromatic number of the co-occurrence graph $\Gamma_{t}$ is upper bounded by $q_{min}$. 
\begin{lemma} 
\label{lem:OmegaGtLThanQmin}
The function  $t$ as defined above 
is a $k$-representative function obtained from some $k$-SCF colouring and $\chi(\Gamma_{t}) \leq q_{min}$.

\end{lemma}
\begin{proof} 
Since $S_{min}$ is an exact-$k$-hitting set of hyperedge cliques, it follows that for every hyperedge $I \in \mathcal{I}$, there exists exactly $\min\{|I|, k\}$ nodes in $S_{min}$ whose hyperedge coordinate is $I$. 
Since every interval is assigned exactly $\min\{|I|, k\}$ representatives by $t$, it follows from the proof of Theorem \ref{thm:kCo-occChar} that any proper colouring of $\Gamma_{t}$ is a $k$-SCF colouring of $H$. Therefore, $t$ is a $k$-representative function obtained from such a $k$-SCF colouring of $H$.
Now, we show that $\chi(\Gamma_{t}) \leq q_{min}$.  In Theorem \ref{thm:Co-occPerf}, we showed that $\Gamma_{t}$ is a perfect graph. It follows from property P1 of perfect graphs in Section \ref{sec:Prelims} that the clique number $\omega$ and the chromatic number $\chi$ of every induced subgraph of $\Gamma_{t}$ are equal. Hence it suffices to show that $\omega(\Gamma_{t}) \leq q_{min}$. Further, since every colour clique is hit at most $q_{min}$ times by $S_{min}$, it follows that the clique number of $\hat{G}[S_{min}]$ is at most $q_{min}$. Hence it is sufficient to show that $\omega(\Gamma_{t}) \leq \omega(\hat{G}[S_{min}])$.  In particular, for each clique in $\Gamma_{t}$ we identify a clique of the same size in $\hat{G}[S_{min}]$.  

The proof is by induction on the size of a clique in $\Gamma_{t}$. The base case is for a clique of size $1$ in $\Gamma_{t}$.  Clearly, there is a clique of size 1 in $\hat{G}[S_{min}]$.   By the induction hypothesis, corresponding to a clique comprising of $u_1,u_2,\ldots,u_{q-1}$ in $\Gamma_{t}$, there is a clique containing nodes $(I_1,u_1),(I_2,u_2),\ldots,(I_{q-1},u_{q-1})$ in $\hat{G}[S_{min}]$. Now, we prove the claim when there are $q$ vertices in a clique in $\Gamma_{t}$. Let $u_1,u_2,\ldots,u_q$ be the set of vertices in the clique. 
Without loss of generality, assume that $u_1,u_2,\ldots,u_{q-1},u_q$ is the left to right ordering of points on the line. Since $(u_1, u_q) \in E(\Gamma_{t})$, there exists an interval, say $I'$ such that $u_1$ and $u_q$ belong to $I'$ and $t(I') \cap \{u_1,u_q\} \neq \emptyset$.   We prove the claim for the case when $u_1 \in t(I')$.  It follows that the node $(I',u_1) \in S_{min}$.  Since both $u_1$ and $u_q$ belong to the interval $I'$, it follows that $u_2,\ldots,u_{q-1}$ also belong to interval $I'$. By the induction hypothesis, for the $q-1$ sized clique $u_2, u_3, \ldots, u_q$ in $\Gamma_{t}$, there is a clique containing the nodes $(I_2,u_2),(I_3,u_3),\ldots,(I_{q},u_{q})$ in $\hat{G}[S_{min}]$.   Therefore,  it follows that $(I',u_1)$ is adjacent to all the nodes $(I_2,u_2),(I_3,u_3),\ldots,(I_{q},u_{q})$ in $\hat{G}$. It follows that corresponding to the clique $u_1, \ldots, u_q$ in $\Gamma_{t}$, there is a clique $(I',u_1),(I_2,u_2),\ldots,(I_{q},u_{q})$ in $\hat{G}[S_{min}]$.   
In case $u_q \in t(I')$, an identical argument is applied to the clique $u_1, \ldots, u_{q-1}$ in $\Gamma_{t}$ to prove the claim.  Hence the lemma.
\qed
\end{proof}
We next show that finding a $k$-SCF colouring using minimum colours is equivalent to finding an exact-$k$-hitting set of hyperedge cliques such that colour cliques are hit as few times as possible. 
\begin{lemma}
\label{lem:khitsetequiv}
There exists a set $S \subseteq V(\hat{G})$ such that 
\begin{itemize}
\item for each $Q \in \mathcal{Q}_1$, $|S \cap Q| = \min\{|I|, k\}$, where $Q$ is the clique corresponding to interval $I$
\item for each $Q' \in \mathcal{Q}_2$, $|S \cap Q'| \leq q$

\end{itemize}
if and only if there is a $k$-SCF colouring of $H$ with $q$ colours.

\end{lemma}
\begin{proof} Let $S$ be a subset of $V(\hat{G})$ such that for each $Q \in \mathcal{Q}_1$, $|S \cap Q| = \min\{|I|, k\}$, where $Q$ is the clique corresponding to interval $I$ and for each $Q' \in \mathcal{Q}_2$, $|S \cap Q'| \leq q$. Then by Lemma \ref{lem:OmegaGtLThanQmin}, there exists a $k$-representative function $t$ such that $\chi(\Gamma_{t}) \leq q$.  It further follows from Theorem \ref{thm:kCo-occChar} that a proper colouring of $\Gamma_{t}$ is a $k$-SCF colouring of $H$ using $\chi(\Gamma_{t}) \leq q$ colours.  This completes the forward direction of the claim.\\
Next we prove the reverse direction. Let $C$ be a $k$-SCF colouring of $H$ using $q$ colours. Then by Theorem \ref{thm:kCo-occChar}, $C$ gives a $k$-representative function $t$ with the property $\chi(\Gamma_{t}) \leq q$.   Since each co-occurrence graph of an interval hypergraph is perfect by Theorem \ref{thm:Co-occPerf}, it follows that $q \geq \chi(\Gamma_{t}) = \omega(\Gamma_{t})$.  Define $S \triangleq \{(I,u)  \mid I \in \mathcal{I}, u \in t(I)\}$. By Theorem \ref{thm:Co-OccConflRelation}, the set $S$ is an exact-$k$-hitting set of cliques in $\mathcal{Q}_1$ and $\omega(\hat{G}[S]) \leq \omega(\Gamma_{t}) \leq q$.   Thus we conclude that if there is a $k$-SCF colouring of $H$ using $q$ colours, then there exists an exact-$k$-hitting set of cliques in $\mathcal{Q}_1$ that intersects every maximal clique in $\mathcal{Q}_2$ at most $q$ times.
\qed
\end{proof}
Lemma \ref{lem:khitsetequiv}  naturally results in an LP formulation in Section \ref{subsec:LP} to solve the problem of finding an exact-$k$-hitting set of hyperedge cliques such that colour cliques are hit as few times as possible.
\subsection{Linear Program for exact-\texorpdfstring{$k$}{k}-hitting sets of hyperedge cliques} 
\label{subsec:LP}
Let $H = (\mathcal{V},\mathcal{I})$ be the input interval hypergraph. We obtain an exact-$k$-hitting set of hyperedge cliques of $\hat{G}(H)$, that hits each colour clique as few times as possible, through an LP formulation. In this LP, there is one variable corresponding to each node of $\hat{G}$. Additionally, $q > 0$ is a fixed integer value that remains constant throughout the execution of the LP. When the LP is run for the first time, $q$ is initialized to $1$. If the LP does not have a feasible solution for the current value of $q$, then $q$ is incremented by $1$ and the LP is executed using the new value of $q$. From our earlier results and the results presented in the following sections, it will become clear that the maximum of $k$ and the smallest value of $q$ for which a feasible solution is obtained for the LP is, in fact, the $k$-SCF colouring number of $H$. \\
First, we present the linear program. Define $X \triangleq \{ x_{I,u} \mid (I,u) \in \hat{G}\}$ to be the set of variables in the LP, where
\[
x_{I,u} = 
\begin{cases}
1, &\quad\text{if node } (I,u)  \text{ hits hyperedge clique corresponding to } I\\
0, &\quad\text{otherwise}
\end{cases} 
\]
\noindent
\textbf{LP Formulation. }\\
Find values to variables in the set $\displaystyle \{x_{I,u} \mid u \in I, I \in \mathcal{I}\}$  subject to \\
\begin{minipage}{\linewidth}
 \begin{align}
&\displaystyle\sum\limits_{u \in I} x_{I,u} = \min\{|I|, k\}, \forall I \in 	\mathcal{I} 
\label{eqn:ehs3} \\
&\displaystyle\sum\limits_{(I,u) \in Q} x_{I,u} \leq q,  \text{ for each maximal clique }Q \text{ in } \mathcal{Q}_2
\label{eqn:ehs4} \\
&x_{I,u} \leq 1
\label{eqn:ehs5} \\
&q > 0 \nonumber
\end{align}
\end{minipage}

The LP has a set of equations (Equation \ref{eqn:ehs3}) and a set of inequalities (Equations \ref{eqn:ehs4} and \ref{eqn:ehs5}). Logically, an equation corresponds to choosing exactly $\min\{|I|, k\}$ vertices per interval $I$; that is, an equation corresponding to interval $I$ corresponds to choosing exactly $\min\{|I|, k\}$ nodes from the clique corresponding to $I$ in $\mathcal{Q}_1$. On the other hand, an inequality in Equation \ref{eqn:ehs4} corresponds to a maximal clique in $\mathcal{Q}_2$. Logically, the inequality means that we pick at most $q$ nodes from every maximal clique in $\mathcal{Q}_2$. Together, any integral solution to the LP is an exact-$k$-hitting set of cliques in $\mathcal{Q}_1$ such that each maximal clique in $\mathcal{Q}_2$ is hit at most $q$ times. 
This LP relaxation is solved using the ellipsoid method \cite{Grotschel1981} which uses a polynomial time separation oracle that we next design. Let $x$ denote an optimum solution to the LP relaxation.  In Section \ref{subsubsec:RoundingLP}, we present a rounding technique that converts the fractional solution $x$ to a feasible integer solution for the LP in polynomial time.

\subsection{Separation Oracle based LP Algorithm \texorpdfstring{$\mathtt{SPAlg}$}{SPAlg}} 
\label{subsec:kSepOracle}
\begin{algorithm}[h]
\caption{\texttt{SPAlg}$(\hat{G},k)$}
\label{algo:SPAlg}
\begin{algorithmic}[1]
\STATE{$q \leftarrow 1$}
\STATE{$\mathcal{B} \leftarrow $ LP instance corresponding to $\hat{G},k$ with $q = 1$}
\WHILE{not \texttt{SPMaxWtClique}($\mathcal{B},q$)}
\STATE{$q \leftarrow q+1$}
\STATE{$\mathcal{B} \leftarrow $ LP instance corresponding to $\hat{G},k$ with new value of $q$}
\ENDWHILE
\STATE{$q_{min} \leftarrow q$}
\STATE{$X_{opt} \leftarrow $ solution for LP instance $\mathcal{B}$ on $q_{min}$}
\RETURN $X_{opt},q_{min}$
\end{algorithmic}
\end{algorithm}
Algorithm $\mathtt{SPAlg}$ (Algorithm \ref{algo:SPAlg}) returns an optimal solution for the LP in $\P$-time using the ellipsoid method which repeatedly invokes a polynomial time separation oracle, which we refer to as \texttt{SPMaxWtClique} (Algorithm \ref{algo:SPMaxWtClique}). 
We describe \texttt{SPMaxWtClique} for a fixed positive integer value $q$ below. 
\begin{algorithm}[h]
\caption{\texttt{SPMaxWtClique}$(\mathcal{B},q)$}
\label{algo:SPMaxWtClique}
\begin{algorithmic}[1]
\STATE{$\hat{G}^w \leftarrow$ vertex-weighted graph of $\hat{G}$ defined as follows: \\for each $(I,v) \in V(\hat{G}^w)$, $w\big((I,v)\big) \leftarrow x_{I,v}$}
\STATE{$Q_{max} \leftarrow $maximum weight clique of $\hat{G}^w$}
\IF{$w(Q_{max}) > q$}
\RETURN{NO}
\hfill
$\blacktriangleright$ \textit{\small{The inequality corresponding to $Q_{max}$ is the violated inequality}}
\ELSIF{any equality is violated}
\RETURN{NO} \hfill
$\blacktriangleright$ \textit{\small{There is a violated equality}}
\ELSE
\RETURN{YES}
\ENDIF
\end{algorithmic}
\end{algorithm}
For each $(I,v)\in V(\hat{G})$, let $x_{I,v}$ be a rational value assigned to the corresponding variable in the LP relaxation.   Given this as an input, for a fixed positive integer value $q$, the separation oracle  \texttt{SPMaxWtClique}
considers the vertex-weighted graph $\hat{G}^w$ corresponding to $\hat{G}$, where the weight of node $(I,v)$ is $x_{I,v}$ for all $(I,v)\in V(\hat{G})$.
The oracle then computes the maximum weight clique of $\hat{G}^w$. 
If the weight of the maximum weight clique of $\hat{G}^w$ exceeds $q$, then it follows that there is some maximal clique $Q'$ whose weight is more than $q$. 
This implies that the given point violates  the inequality corresponding to $Q'$, and this inequality is returned by the oracle as the violated inequality.  If the weight of the maximum weight clique is at most $q$, then the oracle checks if all equalities hold. If any equality is violated, then we have a violated constraint, which is returned by the oracle as the violated equality.  If all the constraints are satisfied, then the oracle reports that the given point is feasible. This completes the description of the separation oracle \texttt{SPMaxWtClique}.  We show in Lemma \ref{lem:kSepOracPoly} that \texttt{SPMaxWtClique} runs in polynomial time. 

\begin{lemma} 
\label{lem:kSepOracPoly}
For an input interval hypergraph and for each integer value $q \geq 0$, the separation oracle \texttt{SPMaxWtClique} runs in polynomial time.

\end{lemma}
\begin{proof} 
For an  interval hypergraph $H$, $\hat{G}(H)$ is perfect by Theorem \ref{thm:ConflictPerf}.  From property P4 of perfect graphs listed in Section \ref{sec:Prelims}, it is known that the maximum weight clique problem in perfect graphs can be solved in polynomial time.  Thus we can find the maximum weight clique in the vertex-weighted graph $\hat{G}^w$ .
That is, finding an inequality in the LP corresponding to a maximal clique whose weight exceeds $q$ can  be done in polynomial time. Also, since there are only polynomial number of hyperedge cliques, it follows that detecting the presence of an infeasible equation can also be done in polynomial time.  It follows that \texttt{SPMaxWtClique} runs in polynomial time. 
\qed
\end{proof}
Let $\mathcal{B}$ be an instance of the given LP. We now show that the LP can be solved in polynomial time by the algorithm $\mathtt{SPAlg}$ that takes as inputs the LP instance $\mathcal{B}$ and the variable $q$, and outputs an assignment to the variables in the set $X$. 

\begin{lemma} 
\label{lem:kSPAlgPoly}
For an interval hypergraph, the algorithm $\mathtt{SPAlg}$ runs in polynomial time.

\end{lemma}
\begin{proof} 
We have shown in Lemma \ref{lem:kSepOracPoly} that the separation oracle in \texttt{SPMaxWtClique} runs in polynomial time. 
For each $q$, the ellipsoid method finds a feasible solution in the polytope of $\mathcal{B}$ in polynomial time, using \texttt{SPMaxWtClique}. The number of values of $q$ is at most the number of points in the interval hypergraph $H$.  This is because, from Observation \ref{obs:NodesOfaVertexAreIndep}, for each vertex $u \in {\mathcal V}$ each clique in $\hat{G}$ can contain at most one node whose vertex coordinate is $u$.   Hence the lemma.
\qed
\end{proof}
\noindent
If the solution $X_{opt}$ returned by $\mathtt{SPAlg}$ is integral, then we have an integer solution in polynomial time. If $X_{opt}$ is not integral, then the rounding algorithm presented in the next section returns a feasible integer solution on the input $X_{opt}$. 

\subsection{Rounding the LP solution}
\label{subsubsec:RoundingLP} 
We present Algorithm \texttt{Rounding} (Algorithm \ref{algo:kRounding}) that takes as input the LP solution $X_{opt}$ for the value $q_{min}$ and returns a feasible integer solution.

\begin{algorithm}[ht]
\caption{\texttt{Rounding}($X_{opt},\mathcal{I}'$)}
\label{algo:kRounding}
\begin{algorithmic}[1]
\vspace{3mm}
\STATE{$i \leftarrow 0$}
\STATE{$X_{opt}(0) \leftarrow X_{opt}$}
\WHILE{$\exists x_{I,v} \in X_{opt}(i)$ that does not belong to $\{0,1\}$}
\STATE{$i \leftarrow i + 1$}
\STATE{$X_{opt}(i) \leftarrow X_{opt}(i-1)$} \hfill
$\blacktriangleright$ \textit{\small{Initialize variables for current iteration}}
\STATE{$I_i \leftarrow $ Longest Interval in $\mathcal{I}'$ with the smallest left endpoint }
\STATE{$r \leftarrow r(I_i)$ }
\STATE{$r-1 \leftarrow $ vertex to the immediate left of $r(I_i)$ on the line  }
\FOR{each interval $I'$ that contains $r$ and $r-1$} \label{algLine:roundingBegin}
\IF{$x_{I_i,r}(i-1) \geq 1$} 
\STATE{$x_{I',r-1}(i) \leftarrow x_{I',r-1}(i-1) + (x_{I_i,r}(i-1)-1)$}
\hfill $\blacktriangleright$ \textit{\small{Add excess value to $(I',r-1)$}}
\STATE{$x_{I',r}(i) \leftarrow x_{I',r}(i-1) - (x_{I_i,r}(i-1)-1)$ } 
\hfill $\blacktriangleright$ \textit{\small{Subtract excess value from $(I',r)$}}
\ELSE 
\STATE{$x_{I',r-1}(i) \leftarrow x_{I',r-1}(i-1) + x_{I_i,r}(i-1)$ }
\hfill $\blacktriangleright$ \textit{\small{Add value to $(I',r-1)$}}
\STATE{$x_{I',r}(i) \leftarrow x_{I',r}(i-1) - x_{I_i,r}(i-1)$ }
\hfill $\blacktriangleright$ \textit{\small{Subtract value from $(I',r)$}}
\ENDIF \hfill $\blacktriangleright$ \textit{\small{$x_{I_i,r-1}(i)$ and $x_{I_i,r}(i)$ are also updated in this loop}}
\ENDFOR \label{algLine:roundingEnd}
\STATE{$\mathcal{I}' \leftarrow \mathcal{I}' \setminus I_i$ }
\label{algLine:shortenIntervalBegin}
\STATE{$I_i \leftarrow I_i \setminus r$}
\STATE{$\mathcal{I}' \leftarrow \mathcal{I}' \cup I_i$}
\label{algLine:shortenIntervalEnd}
\hfill
$\triangleright$ \textit{\small{Remove right endpoint of $I_i$}}
\ENDWHILE
\STATE{$X_{optI} \leftarrow X_{opt}(i)$}
\RETURN{$X_{optI}$}
\end{algorithmic}

\end{algorithm}

\noindent
\textbf{Description of Algorithm \texttt{Rounding}: }  Let $x_{I,u}(i)$ denote the value of variable $x_{I,u}$ in the $i^{th}$ iteration. We use $X_{opt}(i)$ to denote the values given to the variables at the beginning of iteration $i$.
At the start of the algorithm, $X_{opt}(1)$ is initialized to $X_{opt}$. In iteration $i$, $I_i$ is the interval with the smallest left endpoint among all intervals of maximum length. 
Denote $r(I_i)$ by $r$ and the point immediately to the left of $r(I_i)$ by $r-1$.  For every other interval $I'$, denote its right endpoint and the point immediately to the left of the right endpoint by $r(I')$ and $r(I') - 1$ respectively. 

We crucially use the fact that in every iteration of the algorithm, at least one variable gets rounded to either $0$ or $1$.
The key steps in rounding procedure are given in steps \ref{algLine:roundingBegin} to \ref{algLine:roundingEnd}. Steps \ref{algLine:shortenIntervalBegin} to \ref{algLine:shortenIntervalEnd} describe the pruning step by which intervals are shortened in every iteration of \textit{while} loop.  
We show in Lemma \ref{lem:kRoundResultsIntegerSoln} that for some positive integer $j$, $X_{opt}(j)$ will be an all integer solution for $\mathcal{B}$, at which step, the algorithm terminates.
\begin{lemma} 
\label{lem:kRoundResultsIntegerSoln}
Let $X_{opt}$ be a fractional feasible solution returned by $\mathtt{SPAlg}(\mathcal{B}, q_{min})$. Then, on the input $X_{opt}$,  \texttt{Rounding} returns, in a polynomial number of steps,  a point in which each variable in $\mathcal{B}$ has a value in the set $\{0,1\}$.

\end{lemma}
\begin{proof} 
From the description of Algorithm \texttt{Rounding}, once the variable $x_{I_i,r}$ becomes either $0$ or $1$, its value does not change in any subsequent iteration.  Additionally, in iteration $i$, since we remove the right endpoint of $I_i$ irrespective of whether $x_{I_i,r}(i)$ is $0$ or $1$, it follows that in the set $\mathcal{I}'$ there is at least one interval whose length has reduced from its length in iteration $i-1$. Hence the number of variables whose value is not 0 or 1 reduces in each iteration.  Therefore every variable in the solution $X_{optI}$ returned by Algorithm \texttt{Rounding} is set to either $0$ or $1$ in at most $\sum_{I \in \mathcal{I}}|I|$ iterations. 
%
\qed
\end{proof}

We next show in Lemma \ref{lem:kfeasibleAfterRounding} that $X_{optI}$, which is the integer solution returned by \texttt{Rounding}, is feasible for the instance $\mathcal{B}$ for the value $q_{min}$.  That is, the values of variables in the equality corresponding to interval $I$ in Equation \ref{eqn:ehs3} add upto $\min\{k,|I|\}$ and those in each inequality in Equation \ref{eqn:ehs4} add up to at most the same value as it was adding up to in $X_{opt}$.  
\begin{lemma} 
\label{lem:kfeasibleAfterRounding}
Let $X_{opt}$ be a fractional feasible solution returned by $\mathtt{SPAlg}(\mathcal{B}, q_{min})$. 
The solution $X_{optI}$ returned by Algorithm \texttt{Rounding} is a feasible solution for the LP instance $\mathcal{B}$ for the value $q_{min}$.

\end{lemma}
\begin{proof} 
First, we prove by induction on the iteration number that Equations \ref{eqn:ehs3} and \ref{eqn:ehs4} are satisfied at the beginning of each iteration. 
The base case is for the first iteration, where we know that $X_{opt}(1)$ is same as  $X_{opt}$  which is feasible for $\mathcal{B}$ for the value $q_{min}$. 
Let us assume that for an integer $i \geq 1$, Equation \ref{eqn:ehs3} and Equation \ref{eqn:ehs4} are satisified for iteration $i-1$. We show that these constraints are satisified for iteration $i$ also.

Observe that $x_{I_i,r}$ is not removed from $X_{opt}(i)$ even though the point $r(I_i)$ is removed in iteration $i$ thus reducing the length of $I_i$ by one. Effectively, removing the right endpoint of interval $I_i$ does not affect the variables in $X_{opt}(i)$. 
Since Equations \ref{eqn:ehs3}, \ref{eqn:ehs4} are satisfied in $X_{opt}(i-1)$ and by rounding, the variable $x_{I,r-1}$ is increased by the exact same value as that which is reduced from variable $x_{I,r}$, it follows that after iteration $i$, all equations in Equation \ref{eqn:ehs3} are satisfied by $X_{opt}(i)$. In particular, the values of variables corresponding to interval $I$ add up to $\min\{|I|, k\}$. Eventually, in equation corresponding to interval $I$, there will be $\min\{|I|, k\}$ variables with value $1$ and all others with value $0$.

Now, we show that the inequalities in Equation \ref{eqn:ehs4} corresponding to the maximal cliques are also satisfied by $X_{opt}(i)$.  Let $I'$ be an interval that contains the point $r-1$ such that $x_{I',r-1}$ has increased due to steps \ref{algLine:roundingBegin} to \ref{algLine:roundingEnd} in Algorithm \ref{algo:kRounding}. By the choice of $I'$ for which $x_{I',r-1}$ is increased, it follows that $x_{I',r}$ is reduced and thus $I'$ contains the point $r$. 
It follows from the definition  of the edge set  $E_{colour}$ that there is an edge between $(I',r-1)$ and $(I_i,r)$ in $\hat{G}$. 

Let $Q$ be a maximal clique that contains the node $(I',r-1)$. By Observation \ref{obs:NodesOfaVertexAreIndep}, all nodes with the same vertex coordinate form an independent set. Hence $Q$ does not contain any node of the form $(I'',r-1)$, where $I'' \neq I'$. Therefore, there is at most one node in $Q$ whose value increases. If $Q$ contains the node $(I',r)$, then $x_{I',r}$ has reduced and hence the inequality corresponding to $Q$ is satisfied under $X_{opt}(i)$. 
If $Q$ does not contain the node $(I',r)$ we now show that it must contain a node whose vertex coordinate is $r$.  To prove this, among all nodes in $Q$, consider two nodes - one for which the vertex coordinate is leftmost and another for which the vertex coordinate is the rightmost on the line. We denote the leftmost coordinate by $\lambda$ and the rightmost coordinate by $\rho$. Let $(J,\lambda)$ and $(J',\rho)$ be  two nodes in $Q$. 

First, we show that $\lambda \geq l(I_i)$. That is, the point $\lambda$ lies at or to the right of point $l(I_i)$ on the line. The proof is by contradiction. Suppose $\lambda < l(I_i)$. Due to the edge between nodes $(J,\lambda)$ and $(I',r-1)$ in $Q$, it is clear that either $J$ or $I'$ contains both $\lambda$ and $r-1$. Without loss of generality, assume that $J$ contains both $\lambda$ and $r-1$. Since by our assumption $\lambda < l(I_i)$, it follows that $J$ is at least as long as $I_i$ and $l(J) < l(I_i)$. This is a contradiction to our choice of $I_i$ being the longest interval with the smallest left endpoint. It follows that $\lambda \geq l(I_i)$. 
We show using the following cases that the inequality corresponding to $Q$ is still feasible.
\begin{enumerate} 
\item Case $\rho < r-1$. We show that this case is not possible. Since $(I',r-1)$ belongs to $Q$, and $\rho$ is the rightmost vertex coordinate among all nodes in $Q$, it follows that $\rho \geq r-1$. 
\item Case $\rho = r-1$. 
Since $\lambda \geq l(I_i)$ and $\rho = r-1$, it follows that all points from $\lambda$ to $\rho$ belong to $I_i$.  Therefore, by the definition of the edges of $\hat{G}$, $(I_i, r)$ is adjacent to all the nodes of $Q$.  This contradicts the premise that $Q$ is a maximal clique that does not contain $(I_i,r)$. Therefore $\rho = r-1$ is not possible.  
\item Case $\rho = r$. Since $(J',\rho)$ (which is the same as $(J',r)$) belongs to $Q$, it follows that the inequality corresponding to $Q$ is still feasible.  The decrease in $x_{J',r}$ is exactly the same as the increase in $x_{I_i,r-1}$.  
\item Case $\rho > r$. Observe that there is an edge between nodes $(J,\lambda)$ and $(J',\rho)$ since they are both in $Q$. It follows that either $J$ or $J'$ or both contain $\lambda$ and $\rho$. Without loss of generality, let $J$ be this interval. Since $J$ contains all the points on the line from $\lambda$ to $\rho$, both included, it follows that the interval $J$ contains both points $r$ and $r-1$. Further, by the definition of the graph $\hat{G}$, it follows that $(J,r)$ is adjacent to all the nodes in $Q$ whose vertex coordinates which are different from $r$ and lie between $\lambda$ and $\rho$, both included.  Also, since there can be at most one node in a maximal clique with any particular vertex coordinate, and since $Q$ is a maximal clique, it follows that  either $(J,r)$ belongs to $Q$ or that $Q$ contains a node $(J'',r)$ where $J \neq J''$. Since $x_{I_i,r}$ is reduced in iteration $i$, follows that $x_{J,r}$ and $x_{J'',r}$ are also reduced. Therefore, in the maximal clique $Q$ the increase in $x_{I',r-1}$ is compensated by a decrease in $x_{J,r}$ or $x_{J'',r}$ whichever is present in $Q$. Hence, the inequality corresponding to $Q$ is satisfied in by $X_{opt}(i)$.  
\end{enumerate}
This completes the proof of the induction hypothesis that $X_{opt}(i)$ satisfies Equation \ref{eqn:ehs3} and Equation \ref{eqn:ehs4}.  Therefore, on termination of \texttt{Rounding}, $X_{optI}$ satisfies Equation \ref{eqn:ehs3} and Equation \ref{eqn:ehs4}.
Further, by Lemma \ref{lem:kRoundResultsIntegerSoln}, it is clear that \texttt{Rounding} returns an integer point starting with $X_{opt}$ in a polynomial number of steps. That is, for each $I,u$, $x_{I,u}$ is either a $0$ or a $1$ at the end of \texttt{Rounding}. 
Therefore, Equation \ref{eqn:ehs5} is also satisfied by $X_{optI}$. Hence the lemma.
\qed
\end{proof}

Finally, we prove the main result in this paper. 
\begin{theorem} \label{thm:kCFCIntHypPolyTime}
The $k$-SCF colouring problem in interval hypergraphs can be solved in polynomial time.

\end{theorem}
\begin{proof}
By Lemma \ref{lem:kSPAlgPoly}, the LP returns a feasible solution in polynomial time using the separation oracle \texttt{SPMaxWtClique}. By Lemma \ref{lem:kRoundResultsIntegerSoln} and Lemma \ref{lem:kfeasibleAfterRounding}, a feasible integer solution can be obtained from the fractional feasible solution in polynomial time. Further, the $k$-representative function $t$ and thereof, the co-occurrence graph $\Gamma_{t}$ can also be obtained in polynomial time. By Theorem \ref{thm:Co-occPerf}, the co-occurrence graph $\Gamma_{t}$ is perfect. Since a proper colouring of a perfect graph can be found in polynomial time, it follows from Theorem \ref{thm:kCo-occChar} that an optimal $k$-SCF colouring of an interval hypergraph can be found in polynomial time. Hence the theorem.
\qed
\end{proof}


\section{Partition into Exactly Hittable Sets and 1-SCF Colouring Number} \label{sec:cfeqehs}
Given a hypergraph $H = (\mathcal{V},\mathcal{E})$, a set $S \subseteq \mathcal{V}$ is an exact hitting set of $H$, if for each $E \in \mathcal{E}$, we have $|S \cap E| = 1$. 
If hypergraph $H$ has an exact hitting set, then we refer to $H$ as an \textit{exactly hittable hypergraph}.
In this section, we show that each interval hypergraph can be partitioned into $1$-SCF colouring number of exactly hittable interval hypergraphs.
This is proved using Lemmas \ref{lem:CFimpliesEHS}, \ref{lem:kPartskClique} and \ref{lem:kClique}. We first show in Lemma \ref{lem:CFimpliesEHS} that for an arbitrary hypergraph, a $1$-SCF colouring with $\ell$ colours gives a partition of the hyperedges into $\ell$ exactly hittable hypergraphs. Then, we prove the other direction for interval hypergraphs: interval hypergraphs that can be partitioned into $\ell$ exactly hittable hypergraphs can be $1$-SCF coloured with $\ell$ colours. \\
In this section, the range of the  $1$-representative function $t$ is the set $\mathcal{V}$. Thus for an edge $E$, $t(E)$ is a vertex. 
\begin{lemma} \label{lem:CFimpliesEHS}
If there exists a $1$-SCF colouring of a hypergraph $H = (\mathcal{V},\mathcal{I})$ with $\ell$ non-zero colours, then there exists a partition of $\mathcal{I}$ into $\ell$ parts $\{\mathcal{I}_1, \ldots, \mathcal{I}_{\ell}\}$ such that each $H_i = (\mathcal{V}, \mathcal{I}_i), 1 \leq i \leq \ell$ is an exactly hittable hypergraph.  
\end{lemma}
\begin{proof}
Given a $1$-SCF colouring $C$ with at most $\ell$ non-zero colours, let $t$ be a representative function $t:\mathcal{I} \rightarrow \mathcal{V}$ such that for each $I \in \mathcal{I}$, $I$ is $1$-SCF coloured by the vertex  $t(I)$.  The hyperedges are partitioned into sets $\{\mathcal{I}_1, \mathcal{I}_2, \ldots \mathcal{I}_{\ell}\}$ based on $t$ and the vertex colouring as follows:  the set $\mathcal{I}_i$ consists of all those hyperedges $I \in \mathcal{I}$ such that the colour of  $t(I)$ is $i$.  We show that  for each $1 \leq i \leq \ell$, $H_i=(\mathcal{V},\mathcal{I}_i)$ is an exactly hittable hypergraph and the exact hitting set is $h_i = \{t(I) \mid I \in \mathcal{I}_i\}$.   
$h_i$ is a hitting set of $\mathcal{I}_i$ because for each $I \in \mathcal{I}_i$, $t(I)$ is in $h_i$. Since all the vertices of $h_i$ have the same colour assigned by $C$, it follows that each $I \in \mathcal{I}_i$ is hit exactly once by $h_i$. Thus, $h_i$ is an exact hitting set of $\mathcal{I}_i$.  Therefore, each $H_i = (\mathcal{V},\mathcal{I}_i)$ is an exactly hittable hypergraph. This proves the  lemma.
\qed
\end{proof}
We next set up the machinery to conclude that if we are given a partition of an interval hypergraph $H$ into $\ell$ exactly hittable interval hypergraphs, then we get a $1$-SCF colouring with at most $\ell$ non-zero colours.
Let $P = \{\mathcal{I}_1, \mathcal{I}_2, \ldots, \mathcal{I}_{\ell}\}$ be a partition of intervals in $\mathcal{I}(H)$, such that each $H_i = (\mathcal{V},\mathcal{I}_i), 1 \leq i \leq \ell$  is an exactly hittable interval hypergraph. We show that there is a $1$-SCF colouring of $H$ with $\ell$ non-zero colours. Let $h_1, \ldots, h_{\ell}$ be the exact hitting sets of the parts $\mathcal{I}_1, \ldots, \mathcal{I}_{\ell}$ respectively. 
For each interval $I \in \mathcal{I}_i$, let  $t(I)$ be the only vertex in $I \cap h_i$.  
Let $\Gamma_t$ be the co-occurrence graph of $H$.  
We now state and prove Lemmas \ref{lem:kPartskClique} and \ref{lem:kClique} and use them in the proof of Theorem \ref{thm:EHS-CF}.  
\begin{lemma} \label{lem:kPartskClique}
Let $Q$ = $\{u_1, \ldots, u_q\}$ be a clique of size $q$ in the co-occurrence graph $\Gamma_t$. Then, there are $q$ distinct parts $s_1, \ldots, s_q$ in the set $\mathcal{P}$ containing intervals $I_1, \ldots, I_q$ respectively, satisfying the following property: for each $u_i$ in $Q$, $u_i$ is the representative of interval $I_i$ and for each edge $u_iu_j$ in $Q$ either $u_j$ is in $I_i$ or $u_i$ is in $I_j$.
\end{lemma}
\begin{proof}
The proof is by induction on the size $q$ of the clique. The claim is true for base case when $q = 1$; then $u_1$ is the representative of some interval $I_1$ in some part $s_1$. Assume that the claim is true for any clique of size $q-1$. Now, we show that the claim is true for clique $Q$ of size $q$. Let $u_1 < \ldots <u_q$ be the left to right ordering of the points (on the line) corresponding to vertices in the clique $Q$. Since $u_1u_q$ is an edge in $Q$, there must exist an interval $I$ such that either $u_1$ or $u_q$ is the representative of $I$ and $u_q$ occurs along with $u_1$ inside $I$. Without loss of generality, assume that $u_1$ is the representative of interval $I$. Observe that $I$ must contain all points in $u_1, \ldots, u_q$. By the induction hypothesis for the points $u_2, \ldots, u_q$, there are parts $s_2, \ldots, s_q$ and intervals $I_2,\ldots,I_q$ such that $u_i$ is representative of $I_i$ and for each edge $u_iu_j$ in the clique on points in $u_2, \ldots, u_q$, either $u_j$ is in $I_i$ or $u_i$ is in $I_j$. We now show that $I$ does not belong to the parts $s_2, \ldots, s_q$ and that it belongs to a different part. Assume for contradiction that $I$ belongs to some part $s_j$ in $\{s_2, \ldots, s_q\}$. Then, the exact hitting set of set $s_j$ contains at least one point $u_j \in Q$ that is distinct from $u_1$. $u_j$ and $u_1$ cannot be the same point because there is an interval $I_j$ in $s_j$ whose representative is $u_j$. Observe that $u_1$ which is the representative of $I$ must also be in the exact hitting set of $s_j$ because according to our assumption, $I$ belongs to $s_j$. Since interval $I$ contains all points in $u_1, \ldots, u_q$, it is hit at least twice by the exact hitting set of $s_j$ which is a contradiction.
\qed
\end{proof}

\begin{lemma} \label{lem:kClique}
The clique number of the co-occurrence graph $\Gamma_t$ is at most $\ell$.
\end{lemma}
\begin{proof}
For a clique of $q$ vertices in $\Gamma_t$, we know from  Lemma \ref{lem:kPartskClique} that there must be $q$ distinct exactly hittable parts $s_1, \ldots s_{q}$ and $q$ intervals $I_1, \ldots I_{q}$ in each part, respectively,  satisfying an additional property which is not important for this argument.  
Therefore, the size of the largest clique in $\Gamma_t$ is at most the number of parts which is at most $\ell$. 
\qed
\end{proof}

\noindent
Now, we present Theorem \ref{thm:EHS-CF}.
\begin{theorem}\label{thm:EHS-CF}
For an interval hypergraph $H = (\mathcal{V},\mathcal{I})$, there exists a partition of $\mathcal{I}$ into $\ell$ parts $\{\mathcal{I}_1, \ldots, \mathcal{I}_{\ell}\}$ such that for each $1 \leq i \leq \ell$, $H_i = (\mathcal{V}, \mathcal{I}_i)$ has an exact hitting set if and only if there exists a $1$-SCF colouring of $H$ with $\ell$ non-zero colours.
\end{theorem}

\begin{proof}
From  Lemma \ref{lem:CFimpliesEHS}, it follows that if there is a $1$-SCF colouring of a hypergraph $H$ with at most $\ell$ non-zero colours, then there is a partition of $H$ into at most $\ell$ exactly hittable hypergraphs.  To prove the other direction, given a partition of interval hypergraph $H$ into $\ell$ exactly hittable interval hypergraphs, we consider $\Gamma_t$ as defined before Lemma \ref{lem:kPartskClique}.   From Lemma \ref{lem:kClique}, the clique number of $\Gamma_t$ is at most $\ell$.  From Theorem \ref{thm:Co-occPerf}, we know that $\Gamma_t$ is a perfect graph. By the Perfect Graph Theorem \cite{Gol2004}, $\chi(\Gamma_t) = \omega(\Gamma_t) \leq \ell$.  Further from Theorem \ref{thm:kCo-occChar}, $\chi_{cf}(H) \leq \chi(\Gamma_t) \leq \ell$.
 Thus, if there exists a partition of interval hypergraph $H$ into $\ell$ exactly hittable interval hypergraphs, then there exists a $1$-SCF colouring of $H$ using at most $\ell$ non-zero colours.  Hence Theorem \ref{thm:EHS-CF} is proved.
 \qed
 \end{proof}

\bibliographystyle{plain}
\bibliography{cfc}
\end{document}